\documentclass[11pt,a4paper]{article}

\usepackage{amsmath}
\usepackage{amsthm}
\usepackage{amssymb}
\usepackage{amscd}
\usepackage[all]{xy}

\title{Cyclic Covers over Strongly Liftable Schemes\footnote{This paper was partially supported
by the National Natural Science Foundation of China (Grant No.\ 11231003 and 11271070), and the
Shanghai Center for Mathematical Sciences, Fudan University.}}
\author{Qihong Xie}
\date{}
\theoremstyle{plain}
\newtheorem{prop}{Proposition}[section]
\newtheorem{lem}[prop]{Lemma}
\newtheorem{thm}[prop]{Theorem}
\newtheorem{cor}[prop]{Corollary}
\newtheorem{conj}[prop]{Conjecture}

\theoremstyle{definition}
\newtheorem{defn}[prop]{Definition}
\newtheorem*{ack}{Acknowledgments}
\newtheorem*{nota}{Notation}

\theoremstyle{remark}

\newtheorem{ex}[prop]{Example}

\newcommand{\Q}{\mathbb Q}

\newcommand{\Z}{\mathbb Z}

\newcommand{\F}{\mathbb F}
\newcommand{\A}{\mathbb A}
\newcommand{\PP}{\mathbb P}
\newcommand{\OO}{\mathcal O}
\newcommand{\II}{\mathcal I}

\newcommand{\LL}{\mathcal L}

\newcommand{\BB}{\mathcal B}

\newcommand{\cA}{\mathcal A}

\newcommand{\Supp}{\mathop{\rm Supp}\nolimits}
\newcommand{\Sing}{\mathop{\rm Sing}\nolimits}

\newcommand{\ch}{\mathop{\rm char}\nolimits}

\newcommand{\Spec}{\mathop{\bf Spec}\nolimits}
\newcommand{\spec}{\mathop{\rm Spec}\nolimits}

\newcommand{\proj}{\mathop{\rm Proj}\nolimits}

\newcommand{\tor}{\mathop{\rm Tor}\nolimits}

\newcommand{\divisor}{\mathop{\rm div}\nolimits}

\newcommand{\ra}{\rightarrow}

\newcommand{\inj}{\hookrightarrow}
\newcommand{\wt}{\widetilde}

\begin{document}

\maketitle

\begin{abstract}
A smooth scheme $X$ over a field $k$ of positive characteristic
is said to be strongly liftable over $W_2(k)$, if $X$ and all prime
divisors on $X$ can be lifted simultaneously over $W_2(k)$.
In this paper, we give a criterion for that cyclic covers
over strongly liftable schemes are still strongly liftable.
As a corollary, cyclic covers over projective spaces of dimension
at least three are strongly liftable over $W_2(k)$.
\end{abstract}

\setcounter{section}{0}
\section{Introduction}\label{S1}

Throughout this paper, we always work over {\it an algebraically
closed field $k$ of characteristic $p>0$} unless otherwise stated.
A smooth scheme $X$ is said to be strongly liftable over $W_2(k)$,
if $X$ and all prime divisors on $X$ can be lifted simultaneously
over $W_2(k)$. This notion was first introduced in \cite{xie10} to
study the Kawamata-Viehweg vanishing theorem in positive characteristic,
furthermore, many examples and properties of strongly liftable schemes
were given in \cite{xie10,xie11,xw13}.

Before stating the main theorem, let us fix some notation and assumptions.

Let $X$ be a smooth projective variety, and $\LL$ an invertible sheaf on $X$. Let $N$ be a positive integer
prime to $p$, $0\neq s\in H^0(X,\LL^N)$, and $D=\divisor_0(s)$ the effective divisor of zeros of $s$.
Let $\cA=\bigoplus^{N-1}_{i=0}\LL^{-i}(\big[\displaystyle\frac{iD}{N}\big])$, $Y=\Spec\cA$,
and $\pi:Y\ra X$ the cyclic cover obtained by taking the $N$-th root out of $s$.

Assume that $X$ is strongly liftable over $W_2(k)$, $H^1(X,\LL^N)=0$ and $\Sing(D_{\rm red})=\emptyset$.
By \cite[Theorem 4.1 and Corollary 4.3]{xie11}, $X$ has a lifting $\wt{X}$ over $W_2(k)$,
$\LL$ has a lifting $\wt{\LL}$ on $\wt{X}$, $s$ has a lifting $\wt{s}\in H^0(\wt{X},\wt{\LL}^N)$,
and $Y$ is a smooth projective scheme which is liftable over $W_2(k)$.

In this paper, we shall give a criterion for that cyclic covers over strongly liftable schemes
are still strongly liftable (see \S \ref{S3} and \S \ref{S4} for more details).

\begin{thm}\label{1.1}
With the same notation, assumptions and liftings $\wt{X},\wt{\LL}$ and $\wt{s}$ as above,
assume further that for any prime divisor $E$ on $X$ which is not contained in $\Supp(D)$,
there exists a lifting $\wt{E}\subset\wt{X}$ of $E\subset X$ such that
$\wt{s}|_{\wt{E}}\in H^0(\wt{E},\wt{\LL}^N|_{\wt{E}})$ is a divisible lifting of $s|_E\in H^0(E,\LL^N|_E)$.
Then $Y$ is strongly liftable over $W_2(k)$.
\end{thm}

As a consequence of Theorem \ref{1.1}, we have the following corollaries.

\begin{cor}\label{1.2}
Let $X$ be a smooth projective variety satisfying the $H^i$-vanishing condition for $i=1,2$.
Then $X$ is strongly liftable over $W_2(k)$. Let $\LL$ be an invertible sheaf on $X$,
$N$ a positive integer prime to $p$, and $D$ an effective divisor on $X$ with $\LL^N=\OO_X(D)$
and $\Sing(D_{\rm red})=\emptyset$. Let $\pi:Y\ra X$ be the cyclic cover obtained by taking
the $N$-th root out of $D$. Then $Y$ is a smooth projective scheme which is strongly liftable
over $W_2(k)$.
\end{cor}

\begin{cor}\label{1.3}
Let $X=\PP^n_k$ with $n\geq 3$, and $\LL$ an invertible sheaf on $X$.
Let $N$ be a positive integer prime to $p$, and $D$ an effective divisor
on $X$ with $\LL^N=\OO_X(D)$ and $\Sing(D_{\rm red})=\emptyset$.
Let $\pi:Y\ra X$ be the cyclic cover obtained by taking the $N$-th root out of $D$.
Then $Y$ is a smooth projective scheme which is strongly liftable over $W_2(k)$.
\end{cor}

In \S \ref{S2}, we will recall some definitions and preliminary results of strongly liftable schemes.
In \S \ref{S3}, we will give some preliminary results of cyclic covers. The main theorem will be proved
in \S \ref{S4}. For the necessary notions and results on the cyclic cover trick, we refer the reader to \cite{ev}.

\begin{nota}
We use $[B]=\sum [b_i] B_i$ (resp.\ $\ulcorner B\urcorner=\sum \ulcorner b_i\urcorner B_i$,
$\langle B\rangle=\sum \langle b_i\rangle B_i$) to denote the round-down (resp.\ round-up,
fractional part) of a $\Q$-divisor $B=\sum b_iB_i$, where for a real number $b$,
$[b]:=\max\{ n\in\Z \,|\,n\leq b \}$, $\ulcorner b\urcorner:=-[-b]$ and $\langle b\rangle:=b-[b]$.
We use $\Sing(D_{\rm red})$ (resp.\ $\Supp(D)$) to denote the singular locus of the
reduced part (resp.\ the support) of a divisor $D$.
\end{nota}

\begin{ack}
I would like to express my gratitude to Professor Luc Illusie, Professor H\'el\`ene Esnault and
the referees for many useful comments, which make this paper more readable.
\end{ack}

\section{Preliminaries on strongly liftable schemes}\label{S2}

\begin{defn}\label{2.1}
Let $W_2(k)$ be the ring of Witt vectors of length two of $k$.
Then $W_2(k)$ is flat over $\Z/p^2\Z$, and $W_2(k)\otimes_{\Z/p^2\Z}\F_p=k$.
The following definition \cite[Definition 8.11]{ev} generalizes the definition
\cite[1.6]{di} of liftings of $k$-schemes over $W_2(k)$.

Let $X$ be a noetherian scheme over $k$, and $D=\sum D_i$ a reduced Cartier
divisor on $X$. A lifting of $(X,D)$ over $W_2(k)$ consists of a scheme
$\wt{X}$ and closed subschemes $\wt{D}_i\subset\wt{X}$, all defined and
flat over $W_2(k)$ such that $X=\wt{X}\times_{\spec W_2(k)}\spec k$ and
$D_i=\wt{D}_i\times_{\spec W_2(k)}\spec k$. We write
$\wt{D}=\sum \wt{D}_i$ and say that $(\wt{X},\wt{D})$ is a lifting
of $(X,D)$ over $W_2(k)$, if no confusion is likely.

Let $\LL$ be an invertible sheaf on $X$. A lifting of $(X,\LL)$ consists
of a lifting $\wt{X}$ of $X$ over $W_2(k)$ and an invertible sheaf $\wt{\LL}$
on $\wt{X}$ such that $\wt{\LL}|_X=\LL$. For simplicity, we say that
$\wt{\LL}$ is a lifting of $\LL$ on $\wt{X}$, if no confusion is likely.
\end{defn}

Let $\wt{X}$ be a lifting of $X$ over $W_2(k)$. Then $\OO_{\wt{X}}$ is flat
over $W_2(k)$, hence flat over $\Z/p^2\Z$. Note that there is an exact
sequence of $\Z/p^2\Z$-modules:
\[
0\ra p\cdot\Z/p^2\Z\ra \Z/p^2\Z\stackrel{r}{\ra} \Z/p\Z\ra 0,
\]
and a $\Z/p^2\Z$-module isomorphism $p:\Z/p\Z\ra p\cdot\Z/p^2\Z$.
Tensoring the above by $\OO_{\wt{X}}$, we obtain an exact sequence of
$\OO_{\wt{X}}$-modules:
\begin{eqnarray}
0\ra p\cdot\OO_{\wt{X}}\ra \OO_{\wt{X}}\stackrel{r}{\ra}
\OO_X\ra 0, \label{es1}
\end{eqnarray}
and an $\OO_{\wt{X}}$-module isomorphism
\begin{eqnarray}
p:\OO_X\ra p\cdot\OO_{\wt{X}}, \label{es2}
\end{eqnarray}
where $r$ is the reduction modulo $p$ satisfying $p(x)=p\wt{x}$,
$r(\wt{x})=x$ for $x\in\OO_X$, $\wt{x}\in\OO_{\wt{X}}$.

\begin{defn}\label{2.2}
Let $X$ be a smooth scheme over $k$. $X$ is said to be strongly liftable
over $W_2(k)$, if there is a lifting $\wt{X}$ of $X$ over $W_2(k)$, such that
for any prime divisor $D$ on $X$, $(X,D)$ has a lifting $(\wt{X},\wt{D})$
over $W_2(k)$ as in Definition \ref{2.1}, where $\wt{X}$ is fixed for all
liftings $\wt{D}$.
\end{defn}

Let $X$ be a smooth scheme over $k$, $\wt{X}$ a lifting of $X$ over $W_2(k)$,
$D$ a prime divisor on $X$ and $\LL_D=\OO_X(D)$ the associated invertible
sheaf on $X$. Then there is an exact sequence of abelian sheaves:
\begin{eqnarray}
0\ra \OO_X\stackrel{q}{\ra} \OO^*_{\wt{X}}\stackrel{r}{\ra}\OO^*_X\ra 1,
\label{es3}
\end{eqnarray}
where $q(x)=p(x)+1$ for $x\in\OO_X$, $p:\OO_X\ra p\cdot\OO_{\wt{X}}$ is
the isomorphism (\ref{es2}) and $r$ is the reduction modulo $p$. The exact
sequence (\ref{es3}) gives rise to an exact sequence of cohomology groups:
\begin{eqnarray}
H^1(\wt{X},\OO^*_{\wt{X}})\stackrel{r}{\ra} H^1(X,\OO^*_X)\ra
H^2(X,\OO_X). \label{es4}
\end{eqnarray}
If $r:H^1(\wt{X},\OO^*_{\wt{X}})\ra H^1(X,\OO^*_X)$ is surjective,
then $\LL_D$ has a lifting $\wt{\LL}_D$. We combine (\ref{es1}) and (\ref{es2})
to obtain an exact sequence of $\OO_{\wt{X}}$-modules:
\begin{eqnarray}
0\ra \OO_{X}\stackrel{p}{\ra} \OO_{\wt{X}}\stackrel{r}{\ra}
\OO_X\ra 0. \label{es5}
\end{eqnarray}
Tensoring (\ref{es5}) by $\wt{\LL}_D$, we have an exact sequence of
$\OO_{\wt{X}}$-modules:
\begin{eqnarray*}
0\ra \LL_D\stackrel{p}{\ra} \wt{\LL}_D\stackrel{r}{\ra} \LL_D\ra 0,
\end{eqnarray*}
which gives rise to an exact sequence of cohomology groups:
\begin{eqnarray}
H^0(\wt{X},\wt{\LL}_D)\stackrel{r}{\ra} H^0(X,\LL_D)\ra H^1(X,\LL_D).
\label{es6}
\end{eqnarray}

There is a criterion for strong liftability over $W_2(k)$ \cite[Proposition 2.5]{xie11}.

\begin{prop}\label{2.3}
Let $X$ be a smooth scheme over $k$, and $\wt{X}$ a lifting of $X$ over
$W_2(k)$. If for any prime divisor $D$ on $X$, there is a lifting $\wt{\LL}_D$
of $\LL_D=\OO_X(D)$ on $\wt{X}$ such that the natural map $r:H^0(\wt{X},\wt{\LL}_D)\ra
H^0(X,\LL_D)$ is surjective, then $X$ is strongly liftable over $W_2(k)$.
\end{prop}

\section{Preliminaries on cyclic covers}\label{S3}

For convenience of citation, we recall the following result \cite[Theorem 4.1 and Corollary 4.3]{xie11}
with a sketch of the proof.

\begin{thm}\label{3.1}
Let $X$ be a smooth projective variety, and $\LL$ an invertible sheaf on $X$. Let $N$ be a positive
integer prime to $p$, $0\neq s\in H^0(X,\LL^N)$, and $D=\divisor_0(s)$ the divisor of zeros of $s$.
Let $\cA=\bigoplus^{N-1}_{i=0}\LL^{-i}(\big[\displaystyle\frac{iD}{N}\big])$, $Y=\Spec\cA$,
and $\pi:Y\ra X$ the cyclic cover obtained by taking the $N$-th root out of $s$.
Assume that $X$ is strongly liftable over $W_2(k)$, $H^1(X,\LL^N)=0$ and $\Sing(D_{\rm red})=\emptyset$.
Then $X$ has a lifting $\wt{X}$ over $W_2(k)$, $\LL$ has a lifting $\wt{\LL}$ on $\wt{X}$, $s$ has a lifting
$\wt{s}\in H^0(\wt{X},\wt{\LL}^N)$, and $Y$ is a smooth projective scheme which is liftable over $W_2(k)$.
\end{thm}

\begin{proof}
Since $X$ is strongly liftable over $W_2(k)$, there is a lifting $\wt{X}$ of $X$ and a lifting $\wt{\LL}$
of $\LL$ on $\wt{X}$. Since $H^1(X,\LL^N)=0$, the exact sequence (\ref{es6}) gives rise to a surjection
$H^0(\wt{X},\wt{\LL}^N)\stackrel{r}{\ra} H^0(X,\LL^N)$, hence $s$ has a lifting $\wt{s}\in H^0(\wt{X},\wt{\LL}^N)$.
Let $\wt{D}=\divisor_0(\wt{s})$. Then $\wt{D}$ is a lifting of $D$.
Let $\wt{\cA}=\bigoplus^{N-1}_{i=0}\wt{\LL}^{-i}(\big[\displaystyle\frac{i\wt{D}}{N}\big])$
and $\wt{Y}=\Spec\wt{\cA}$. Then $\wt{Y}$ is a lifting of $Y$.
Thus $Y$ is a smooth projective scheme which is liftable over $W_2(k)$.
\end{proof}

The above result says that cyclic covers over strongly liftable schemes are liftable over $W_2(k)$
under certain conditions, however, in general, they are not strongly liftable over $W_2(k)$
(see \cite[Remark 4.6]{xie11} for more details). In order to prove the second part of Theorem \ref{1.1},
some elementary results on cyclic covers over integral schemes are needed.
First of all, we recall an easy lemma \cite[Lemma 3.15(a)]{ev}.

\begin{lem}\label{3.2}
Let $X$ be an integral scheme, and $\LL$ an invertible sheaf on $X$. Let $N$ be a positive integer
prime to $p$, $0\neq s\in H^0(X,\LL^N)$, and $D=\divisor_0(s)$ the divisor of zeros of $s$.
Let $\cA=\bigoplus^{N-1}_{i=0}\LL^{-i}(\big[\displaystyle\frac{iD}{N}\big])$, $Y=\Spec\cA$,
and $\pi:Y\ra X$ the cyclic cover obtained by taking the $N$-th root out of $s$.
Then $Y$ is reducible if and only if there is an integer $\mu>1$ dividing $N$ and a section
$t\in H^0(X,\LL^{N/\mu})$ such that $s=t^{\otimes\mu}$.
\end{lem}

\begin{proof}
We can consider the problem over a dense open subset $\spec B\subset X\setminus D_{\rm red}$.
Since $H^0(X,\LL^N)\cong B$, we may assume that $s\in H^0(X,\LL^N)$ corresponds to an element $u\in B$.
Since $\spec B[x]/(x^N-u)$ is a dense open subset of $Y$, $Y$ is reducible if and only if $x^N-u$ is
reducible in $B[x]$, which is equivalent to the existence of some $v\in B$ with $u=v^\mu$.
\end{proof}

\begin{defn}\label{3.3}
Let $X$ be a scheme, $\LL$ an invertible sheaf on $X$, $N$ a positive integer, and $0\neq s\in H^0(X,\LL^N)$.
The section $s$ is said to be $\mu$-divisible, if $\mu>0$ divides $N$ and there exists a section
$t\in H^0(X,\LL^{N/\mu})$ such that $s=t^{\otimes\mu}$. The section $s$ is said to be maximally $\mu$-divisible,
if $s$ is $\mu$-divisible, and if $s$ is also $\nu$-divisible then $\nu\leq\mu$.
\end{defn}

\begin{lem}\label{3.4}
With notation and assumptions as in Lemma \ref{3.2}, then $Y$ has exactly $\mu$
irreducible components if and only if the section $s$ is maximally $\mu$-divisible.
\end{lem}

\begin{proof}
First of all, we prove that if $s$ is $\mu$-divisible then $Y$ has at least $\mu$ irreducible components. Indeed,
assume that $s$ is $\mu$-divisible, then there is a section $t\in H^0(X,\LL^{N/\mu})$ such that $s=t^{\otimes\mu}$,
$D=\mu D_1$, where $D_1=\divisor_0(t)$. It follows from a direct calculation that $\pi:Y\ra X$ factorizes into the
composition of two cyclic covers: $Y\stackrel{\pi_2}{\ra} Y_1\stackrel{\pi_1}{\ra} X$, where $\pi_1:Y_1\ra X$ is
the cyclic cover obtained by taking the $\mu$-th root out of $1\in H^0(X,(\LL^{N/\mu}(-D_1))^\mu)=H^0(X,\OO_X)$, and
$\pi_2:Y\ra Y_1$ is the cyclic cover obtained by taking the $N/\mu$-th root out of $\pi_1^*t\in H^0(Y_1,\pi_1^*\LL^{N/\mu})$.
Since $\pi_1$ is unramified, $Y_1$ has at least $\mu$ irreducible components, hence so does $Y$.

If the section $s$ is maximally $\mu$-divisible, then $Y$ has at least $\mu$ irreducible components by the above argument.
If $Y$ has exactly $\nu$ irreducible components with $\nu>\mu$, then by the proof of Lemma \ref{3.2}, $s$ is also $\nu$-divisible
with $\nu>\mu$, which is absurd. Conversely, if $Y$ has exactly $\mu$ irreducible components, then $s$ is $\mu$-divisible by
the proof of Lemma \ref{3.2}, and furthermore, $s$ is maximally $\mu$-divisible by the above argument.
\end{proof}

\begin{defn}\label{3.5}
With notation and assumptions as in Definition \ref{3.3}, assume further that $X$ has a lifting $\wt{X}$ over $W_2(k)$,
$\LL$ has a lifting $\wt{\LL}$ on $\wt{X}$. A section $\wt{s}\in H^0(\wt{X},\wt{\LL}^N)$ is called a divisible
lifting of $s\in H^0(X,\LL^N)$, if the following conditions hold:
\begin{itemize}
\item[(i)] $\wt{s}$ is a lifting of $s$, i.e.\ $r(\wt{s})=s$; and
\item[(ii)] if there is an integer $\mu>0$ dividing $N$ and a section $t\in H^0(X,\LL^{N/\mu})$ such that $s=t^{\otimes\mu}$,
then there exists a section $\wt{t}\in H^0(\wt{X},\wt{\LL}^{N/\mu})$ lifting $t$ such that $\wt{s}=\wt{t}^{\otimes\mu}$.
\end{itemize}
It is easy to see that if $s$ is maximally $\mu$-divisible and $\wt{s}$ is a divisible lifting of $s$,
then $\wt{s}$ is also maximally $\mu$-divisible.
\end{defn}

\begin{lem}\label{3.6}
With notation and assumptions as in Lemma \ref{3.2}, assume further that $X$ has a lifting $\wt{X}$ over $W_2(k)$,
$\LL$ has a lifting $\wt{\LL}$ on $\wt{X}$, and $s$ has a lifting $\wt{s}\in H^0(\wt{X},\wt{\LL}^N)$. Let
$\wt{D}=\divisor_0(\wt{s})$, $\wt{\cA}=\bigoplus^{N-1}_{i=0}\wt{\LL}^{-i}(\big[\displaystyle\frac{i\wt{D}}{N}\big])$
and $\wt{Y}=\Spec\wt{\cA}$. If $s$ is maximally $\mu$-divisible and $\wt{s}$ is a divisible lifting of $s$, then
$\wt{Y}$ has exactly $\mu$ irreducible components.
\end{lem}

\begin{proof}
By factorizing $\wt{\pi}:\wt{Y}\ra \wt{X}$ into the composition of two cyclic covers, we can prove that $\wt{Y}$ has
at least $\mu$ irreducible components, whose proof is almost identical to the argument given in the proof of Lemma \ref{3.4}
by changing the usual data into the lifted ones. Assume that $\wt{Y}$ has exactly $\nu$ irreducible components with $\nu>\mu$.
Since $\wt{Y}\times_{\spec W_2(k)}\spec k=Y$ and irreducible components of $\wt{Y}$ have distinct underlying topological spaces,
we have that $Y$ has at least $\nu$ irreducible components with $\nu>\mu$, which contradicts Lemma \ref{3.4}. Thus $\wt{Y}$ has
exactly $\mu$ irreducible components.
\end{proof}

\begin{lem}\label{3.7}
With notation and assumptions as in Lemma \ref{3.2}, let $E$ be a prime divisor on $X$ which is not contained in $\Supp(D)$,
$\BB=\bigoplus^{N-1}_{i=0}(\LL|_E)^{-i}(\big[\displaystyle\frac{iD|_E}{N}\big])$, and
$\cA|_E=\bigoplus^{N-1}_{i=0}\LL^{-i}(\big[\displaystyle\frac{iD}{N}\big])|_E$ the restriction of $\cA$ to $E$.
Then there is a natural finite surjective morphism $\tau_E:\Spec\BB\ra \Spec\cA|_E$.
\end{lem}

\begin{proof}
It is easy to see that $\big[\displaystyle\frac{im}{N}\big]\geq m\big[\displaystyle\frac{i}{N}\big]$ holds
for any $i\geq 0$ and $m\geq 1$. Thus there are injective homomorphisms $\LL^{-i}(\big[\displaystyle\frac{iD}{N}\big])|_E
\ra(\LL|_E)^{-i}(\big[\displaystyle\frac{iD|_E}{N}\big])$ for all $0\leq i\leq N-1$, which induce
a natural injective homomorphism of $\OO_E$-algebras: $\bigoplus^{N-1}_{i=0}\LL^{-i}(\big[{\displaystyle\frac{iD}{N}}
\big])|_E\ra \bigoplus^{N-1}_{i=0}(\LL|_E)^{-i}(\big[{\displaystyle\frac{iD|_E}{N}}\big])$. By \cite[(6.D) Lemma 2]{ma80},
there is a natural dominant morphism $\tau_E:\Spec\BB\ra \Spec\cA|_E$, which fits into a commutative diagram:
\[
\xymatrix{
\Spec\BB \ar[rr]^{\tau_E} \ar[rd]_\sigma & & \Spec\cA|_E \ar[ld]^{\pi|_{\pi^{-1}(E)}} \\
 & E &
}
\]
where $\pi|_{\pi^{-1}(E)}:\Spec\cA|_E=E\times_{X}Y=\pi^{-1}(E)\ra E$ is the restriction of $\pi$ to $\pi^{-1}(E)$ over $E$,
and $\sigma:\Spec\BB\ra E$ is the cyclic cover obtained by taking the $N$-th root out of $s|_E$. Since $\BB$ is a finite
$\OO_E$-module, hence a finite $\cA|_E$-module, $\tau_E$ is finite. Since a finite morphism is closed
\cite[Exercise II.3.5]{ha77}, $\tau_E$ is surjective.
\end{proof}

\begin{cor}\label{3.8}
With notation and assumptions as in Lemma \ref{3.7}, assume further that $E$ is smooth. Then $\tau_E:\Spec\BB\ra \Spec\cA|_E$
is the normalization morphism of $\Spec\cA|_E$.
\end{cor}

\begin{proof}
Denote $\cA'|_E=\bigoplus^{N-1}_{i=0}(\LL|_E)^{-i}$. Then there are natural injective homomorphisms of $\OO_E$-algebras:
$\cA'|_E\inj \cA|_E\inj \BB$, which induce morphisms: $\Spec\BB\ra \Spec\cA|_E\ra \Spec\cA'|_E$. Since $E$ is smooth and
$\Spec\BB\ra E$ is the cyclic cover obtained by taking the $N$-th root out of $s|_E$, by \cite[3.5 and 3.10]{ev}, $\Spec\BB$
is the normalization of $\Spec\cA'|_E$, hence of $\Spec\cA|_E$.
\end{proof}

\begin{lem}\label{3.9}
With notation and assumptions as in Lemma \ref{3.6}, let $E$ be a prime divisor on $X$
which is not contained in $\Supp(D)$, $\wt{E}\subset\wt{X}$ a lifting of $E\subset X$,
$\wt{\BB}=\bigoplus^{N-1}_{i=0}(\wt{\LL}|_{\wt{E}})^{-i}(\big[\displaystyle\frac{i\wt{D}|_{\wt{E}}}{N}\big])$, and
$\wt{\cA}|_{\wt{E}}=\bigoplus^{N-1}_{i=0}\wt{\LL}^{-i}(\big[\displaystyle\frac{i\wt{D}}{N}\big])|_{\wt{E}}$ the restriction of
$\wt{\cA}$ to $\wt{E}$. Then there is a natural finite surjective morphism $\tau_{\wt{E}}:\Spec\wt{\BB}\ra \Spec\wt{\cA}|_{\wt{E}}$,
which is a lifting of $\tau_E:\Spec\BB\ra \Spec\cA|_E$ constructed as in Lemma \ref{3.7}.
\end{lem}

\begin{proof}
It is similar to that of Lemma \ref{3.7}.
\end{proof}

We give a simple example to show the difference between $\Spec\BB$ and $\Spec\cA|_E$ defined as in Lemma \ref{3.7}.

\begin{ex}\label{3.10}
Let $X=\PP^2_k=\proj k[x,y,z]$, $\LL=\OO_X(1)$, $N=2$, $s=x^2-yz\in H^0(X,\LL^N)$ with $D=(x^2-yz=0)$, $\ch(k)=p\geq 3$,
and $E=(y=0)$. Consider the cyclic cover $\pi:Y\ra X$ obtained by taking the square root out of $s$. Look at $\pi$ over the
affine piece $\A^2_k=\spec k[u,v]$, where $u=x/z$ and $v=y/z$, then $Y$ is defined by the equation $t^2=u^2-v$, and $E$
is defined by the equation $v=0$. It is easy to see that $\pi^{-1}(E)$ consists of two irreducible components, say $E_1$
and $E_2$, which are defined by the equations $t\pm u=0$ respectively. Thus $E_1$ and $E_2$ are smooth, intersect transversally
and map isomorphically onto $E$.

Since $s|_E$, the restriction of $s$ to $E$, is defined by $x^2=0$ on $E=\proj k[x,z]$, we have $D|_E=2Q$, where $Q$ is the
point $[0:1]$ on $E$. Therefore $\OO_E$-algebras $\BB=\bigoplus^{N-1}_{i=0}(\LL|_E)^{-i}(\big[\displaystyle\frac{iD|_E}{N}\big])
=\OO_E\oplus\OO_E$, $\cA|_E=\bigoplus^{N-1}_{i=0}\LL^{-i}(\big[\displaystyle\frac{iD}{N}\big])|_E=\OO_E\oplus\OO_E(-1)$.
By assumption, $\Spec\cA|_E=E\times_X Y=\pi^{-1}(E)=E_1+E_2$, whereas by Corollary \ref{3.8}, $\Spec\BB=\Spec(\OO_E\oplus\OO_E)
=F_1\coprod F_2$ is a disjoint union of $F_1$ and $F_2$ such that $\tau_E:F_1\coprod F_2\ra E_1+E_2$ is the normalization morphism.
\end{ex}

\section{Proof of the main theorem}\label{S4}

In this section, we shall prove the main theorem as follows.

\begin{thm}\label{4.1}
With notation and assumptions as in Theorem \ref{3.1}, fix such liftings $\wt{X},\wt{\LL}$ and $\wt{s}$ as in Theorem \ref{3.1}.
Assume further that for any prime divisor $E$ on $X$ which is not contained in $\Supp(D)$, there exists a lifting
$\wt{E}\subset\wt{X}$ of $E\subset X$ such that $\wt{s}|_{\wt{E}}\in H^0(\wt{E},\wt{\LL}^N|_{\wt{E}})$ is a divisible
lifting of $s|_E\in H^0(E,\LL^N|_E)$. Then $Y$ is strongly liftable over $W_2(k)$.
\end{thm}

Before proving Theorem \ref{4.1}, we use Example \ref{3.10} to illustrate the meaning of the further assumption made in Theorem \ref{4.1}.

\begin{ex}\label{4.2}
With notation and assumptions as in Example \ref{3.10},
take liftings of $X,\LL,s,D$ and $E$ as follows: $\wt{X}=\PP^2_{W_2(k)}=\proj W_2(k)[x,y,z]$, $\wt{\LL}=\OO_{\wt{X}}(1)$,
$\wt{s}=x^2-yz\in H^0(\wt{X},\wt{\LL}^N)$, $\wt{D}=(x^2-yz=0)$, and $\wt{E}=(y-pz=0)$.
Denote $\wt{Y}=\Spec\bigoplus^{N-1}_{i=0}\wt{\LL}^{-i}(\big[\displaystyle\frac{i\wt{D}}{N}\big])$ and $\wt{\pi}:\wt{Y}\ra \wt{X}$
the induced morphism. Look at $\wt{\pi}$ over the affine piece $\A^2_{W_2(k)}=\spec W_2(k)[u,v]$, where $u=x/z$ and $v=y/z$,
$\wt{Y}$ is defined by $t^2=u^2-v$, and $\wt{E}$ is defined by $v=p$. It is easy to see that $\wt{E}_{12}=\wt{\pi}^{-1}(\wt{E})$
is defined by $t^2=u^2-p$, which is irreducible. Hence by \cite[Lemma 2.2]{xie11}, $\wt{E}_{12}$ is not a lifting of $E_1$ or $E_2$
or $E_1+E_2$.
\[
\xymatrix{
E_1+E_2 \ar[d]_{\pi} \ar@{.>}[r]|\times & \wt{E}_{12} \ar[d]^{\wt{\pi}} \\
E \ar@{^{(}->}[r] & \wt{E}
}
\]

The further assumption made in Theorem \ref{4.1} guarantees that the choices of liftings $\wt{E}$ of $E$ are so adequate that
the above situation can be avoided. In our example, $s|_E$ is maximally $2$-divisible, if we can choose a lifting $\wt{E}$ of $E$
such that $\wt{s}|_{\wt{E}}$ is a divisible lifting of $s|_E$ (so $\wt{s}|_{\wt{E}}$ is also maximally $2$-divisible), then we have
a lifting $\Spec\wt{\BB}=\wt{F}_1\coprod\wt{F}_2$ of $\Spec\BB=F_1\coprod F_2$ such that $\wt{F}_i$ is a lifting of $F_i$ for $i=1,2$.
Let $\wt{E}_i=\tau_{\wt{E}}(\wt{F}_i)$. Then $\wt{E}_i$ is a lifting of $E_i$ for $i=1,2$, since $\tau_{\wt{E}}$ is a lifting of $\tau_E$.
\[
\xymatrix{
F_1\coprod F_2 \ar[d]_{\tau_E} \ar@{^{(}->}[r] & \wt{F}_1\coprod\wt{F}_2 \ar[d]^{\tau_{\wt{E}}} \\
E_1+E_2 \ar[d]_{\pi} \ar@{^{(}->}[r] & \wt{E}_1+\wt{E}_2 \ar[d]^{\wt{\pi}} \\
E \ar@{^{(}->}[r] & \wt{E}
}
\]
\end{ex}

\begin{proof}[Proof of Theorem \ref{4.1}]
Consider the following cartesian square, where $\wt{\pi}:\wt{Y}\ra \wt{X}$ is the natural projection induced
by the definition of $\wt{Y}$ in the proof of Theorem \ref{3.1}:
\[
\xymatrix{
Y \ar[d]_{\pi} \ar@{^{(}->}[r]^\iota & \wt{Y} \ar[d]^{\wt{\pi}} \\
X \ar@{^{(}->}[r]^\iota & \wt{X}.
}
\]

Let $E_Y$ be a prime divisor on $Y$, and $E=\pi_*(E_Y)$ the induced prime divisor on $X$.

If $E\subset\Supp(D)$, then let $\wt{E}\subset\Supp(\wt{D})$ be the corresponding lifting of $E$. We can take an irreducible component
$\wt{E}_Y$ of $\wt{\pi}^{-1}(\wt{E})$ such that $\wt{E}_Y\times_{\wt{E}}E=E_Y$, i.e.\ $\wt{E}_Y$ is a lifting of $E_Y$.

If $E\not\subset\Supp(D)$, then $\pi^{-1}(E)$ may be reducible. Assume that $\pi^{-1}(E)=\sum_{i=1}^\nu E_i$ with $E_1=E_Y$.
Since $E\not\subset\Supp(D)$, $0\neq s|_E\in H^0(E,\LL^N|_E)$ determines the effective divisor $D|_E$ on $E$.
Let $\tau_E:\Spec\BB=\sum_{j=1}^\mu F_j\ra \Spec\cA|_E=\pi^{-1}(E)=\sum_{i=1}^\nu E_i$ be the natural morphism defined as
in Lemma \ref{3.7}, where $F_j$ are distinct irreducible components. Since $\tau_E$ is finite and surjective, we may assume that
$\tau_E(F_1)=E_1$. Since $\Spec\BB=\sum_{j=1}^\mu F_j\ra E$ is the cyclic cover obtained by taking the $N$-th root out of $s|_E$,
by Lemma \ref{3.4}, the section $s|_E$ is maximally $\mu$-divisible. Thus there exists a section
$t_E\in H^0(E,\LL^{N/\mu}|_E)$ such that $s|_E=t_E^{\otimes\mu}$. By assumption, there is a lifting $\wt{E}\subset\wt{X}$ of
$E\subset X$ such that $\wt{s}|_{\wt{E}}$ is a divisible lifting of $s|_E$, i.e.\ there is a section $\wt{t_E}\in
H^0(\wt{E},\wt{\LL}^{N/\mu}|_{\wt{E}})$ lifting $t_E$ such that $\wt{s}|_{\wt{E}}=\wt{t_E}^{\otimes\mu}$.

Consider $\tau_{\wt{E}}:\Spec\wt{\BB}\ra \Spec\wt{\cA}|_{\wt{E}}$ defined as in Lemma \ref{3.9}, where
$\Spec\wt{\BB}\ra \wt{E}$ is the cyclic cover obtained by taking the $N$-th root out of $\wt{s}|_{\wt{E}}$,
and $\Spec\wt{\cA}|_{\wt{E}}=\wt{E}\times_{\wt{X}}\wt{Y}=\wt{\pi}^{-1}(\wt{E})$. Since $s|_E$ is maximally $\mu$-divisible and
$\wt{s}|_{\wt{E}}$ is a divisible lifting of $s|_E$, by Lemma \ref{3.6}, we may assume that $\Spec\wt{\BB}=\sum_{j=1}^\mu \wt{F}_j$,
where $\wt{F}_j$ are distinct irreducible components, hence $\wt{F}_j$ have distinct underlying topological spaces. Since
$\Spec\wt{\BB}\times_{\spec W_2(k)}\spec k=\Spec\BB=\sum_{j=1}^\mu F_j$ and $\wt{F}_j\times_{\spec W_2(k)}\spec k$ are distinct,
up to permutation of indices, we can assume that $\wt{F}_j\times_{\spec W_2(k)}\spec k=F_j$ for any $1\leq j\leq \mu$.

By Lemma \ref{3.9}, $\tau_{\wt{E}}$ is finite and surjective, hence there is an irreducible component of $\wt{\pi}^{-1}(\wt{E})$, say $\wt{E}_1$,
such that $\tau_{\wt{E}}|_{\wt{F}_1}:\wt{F}_1\ra\wt{E}_1$ is surjective. Since $\tau_{\wt{E}}$ is a lifting of $\tau_E$, we have that
$\wt{E}_1\times_{\spec W_2(k)}\spec k=E_1$. Finally, we will show that $\wt{E}_1$ is flat over $W_2(k)$,
whence $\wt{E}_1$ is a lifting of $E_1=E_Y$, thus $Y$ is strongly liftable over $W_2(k)$.
\[
\xymatrix{
F_1+\cdots+F_\mu \ar[d]_{\tau_E} \ar@{^{(}->}[r] & \wt{F}_1+\cdots+\wt{F}_\mu \ar[d]^{\tau_{\wt{E}}} \\
E_1+\cdots+E_\nu \ar[d]_{\pi} \ar@{^{(}->}[r] & \wt{E}_1+\cdots+\wt{E}_\nu \ar[d]^{\wt{\pi}} \\
E \ar@{^{(}->}[r] & \wt{E}
}
\]

Since $W_2(k)$ is an Artin local ring, to prove that $\wt{E}_1$ is flat over $W_2(k)$, by the local criteria of flatness
\cite[(20.C) Theorem 49]{ma80}, it suffices to show $\tor_1^{W_2(k)}(\OO_{\wt{E}_1},k)=0$. Let $Z=\pi^{-1}(E)$, $\wt{Z}=\wt{\pi}^{-1}(\wt{E})$,
$\II$ the ideal sheaf of $E_1$ in $Z$, and $\wt{\II}$ the ideal sheaf of $\wt{E}_1$ in $\wt{Z}$. Then the structure sheaf of $\wt{Z}$ is $\wt{\cA}|_{\wt{E}}$,
which is locally free over $\wt{E}$ and $\wt{E}$ is flat over $W_2(k)$, hence $\wt{Z}$ is flat over $W_2(k)$ and $\wt{Z}\times_{\spec W_2(k)}\spec k=Z$.
Locally, $E_1$ is defined by one of the factors of the equation $x^N=s|_E$, and $\wt{E}_1$ is defined by one of the factors of the equation
$\wt{x}^N=\wt{s}|_{\wt{E}}$. Since $\wt{s}|_{\wt{E}}$ is a divisible lifting of $s|_E$, we have that the reduction of the defining equations
of $\wt{E}_1$ modulo $p$ are just the defining equations of $E_1$, hence $\wt{\II}\times_{\spec W_2(k)}\spec k=\II$ holds.
Considering the following exact sequence:
\begin{eqnarray*}
0\ra \wt{\II}\ra \OO_{\wt{Z}}\ra \OO_{\wt{E}_1}\ra 0,
\end{eqnarray*}
and taking its long exact sequence for $-\otimes_{W_2(k)}k$, we obtain an exact sequence:
\begin{eqnarray*}
0\ra \tor^{W_2(k)}_1(\OO_{\wt{E}_1},k)\ra \wt{\II}\otimes_{W_2(k)}k\ra \OO_Z\ra \OO_{E_1}\ra 0,
\end{eqnarray*}
which implies $\tor_1^{W_2(k)}(\OO_{\wt{E}_1},k)=0$, since $\wt{\II}\otimes_{W_2(k)}k=\II$ and
$0\ra \II\ra \OO_Z\ra \OO_{E_1}\ra 0$ is exact.
\end{proof}

\begin{defn}\label{4.3}
A noetherian scheme $X$ is said to satisfy the $H^i$-vanishing condition, if $H^i(X,\LL)=0$ holds for any invertible sheaf
$\LL$ on $X$. For example, the projective space $\PP^n_k$ satisfies the $H^i$-vanishing condition for any $1\leq i\leq n-1$.
\end{defn}

\begin{cor}\label{4.4}
Let $X$ be a smooth projective variety satisfying the $H^i$-vanishing condition for $i=1,2$. Then $X$ is strongly liftable
over $W_2(k)$, and for any cyclic cover $\pi:Y\ra X$ constructed as in Theorem \ref{3.1}, $Y$ is also strongly liftable
over $W_2(k)$.
\end{cor}

\begin{proof}
From the exact sequences (\ref{es4}) and (\ref{es6}) and Proposition \ref{2.3}, it follows that $X$ is strongly liftable.
By Theorem \ref{4.1}, we have only to show that for any prime divisor $E$ on $X$, there exists a lifting $\wt{E}\subset\wt{X}$
of $E\subset X$ such that $\wt{s}|_{\wt{E}}$ is a divisible lifting of $s|_E$.

Assume that $s|_E\in H^0(E,\LL^N|_E)$ is $\mu$-divisible. Thus there is a section $t_E\in H^0(E,\LL^{N/\mu}|_E)$
such that $s|_E=t_E^{\otimes\mu}$. Take an arbitrary lifting $\wt{E}\subset\wt{X}$ of $E\subset X$ and consider the following
commutative diagram:
\[
\xymatrix{
H^0(\wt{X},\wt{\LL}^{N/\mu}) \ar[d]_{q_{\wt{E}}} \ar@{->>}[r]^r & H^0(X,\LL^{N/\mu}) \ar@{->>}[d]^{q_E} \\
H^0(\wt{E},\wt{\LL}^{N/\mu}|_{\wt{E}}) \ar[r]^r & H^0(E,\LL^{N/\mu}|_E),
}
\]
where the surjectivity of the upper horizontal map $r$ and the right vertical map $q_E$ follows from the $H^1$-vanishing condition
for $X$ by observing the exact sequence (\ref{es6}) and the following exact sequence:
\begin{eqnarray}
0\ra \OO_X(-E)\ra \OO_X\ra \OO_E\ra 0. \label{es7}
\end{eqnarray}
Thus for $t_E\in H^0(E,\LL^{N/\mu}|_E)$, there exists a section $\wt{t}\in H^0(\wt{X},\wt{\LL}^{N/\mu})$ such that $q_E\circ r(\wt{t})=t_E$.
Let $\wt{t_E}=q_{\wt{E}}(\wt{t})$. Then $\wt{t_E}\in H^0(\wt{E},\wt{\LL}^{N/\mu}|_{\wt{E}})$ is a lifting of $t_E$.

The exact sequence (\ref{es7}) gives rise to an exact sequence of cohomology groups:
\begin{eqnarray*}
0=H^1(X,\OO_X)\ra H^1(E,\OO_E)\ra H^2(X,\OO_X(-E))=0,
\end{eqnarray*}
hence we have $H^1(E,\OO_E)=0$. Taking cohomology groups of the following exact sequence,
which is the exact sequence (\ref{es3}) for $\wt{E}$:
\begin{eqnarray*}
0\ra \OO_E\stackrel{q}{\ra} \OO_{\wt{E}}^*\stackrel{r}{\ra} \OO_E^*\ra 0,
\end{eqnarray*}
we have an exact sequence of cohomology groups:
\begin{eqnarray}
0=H^1(E,\OO_E)\ra H^1(\wt{E},\OO_{\wt{E}}^*)\stackrel{r}{\ra} H^1(E,\OO_E^*), \label{es8}
\end{eqnarray}
which implies that $\wt{\LL}_1\cong \wt{\LL}_2$ if and only if $\LL_1\cong \LL_2$, where $\wt{\LL}_i$ are invertible sheaves
on $\wt{E}$ and $\LL_i=\wt{\LL}_i|_E$ for $i=1,2$.

Since $r(\wt{s}|_{\wt{E}})=s|_E=t_E^{\otimes\mu}=r(\wt{t_E}^{\otimes\mu})$, hence there exists a unit $\wt{u}\in\OO_{\wt{E}}^*$
such that $r(\wt{u})=1$ and $\wt{s}|_{\wt{E}}=\wt{u}\wt{t_E}^{\otimes\mu}$. Since $p\nmid N$, we have $p\nmid \mu$, hence
there exists a unit $\wt{v}\in\OO_{\wt{E}}^*$ such that $\wt{v}^\mu=\wt{u}$ and $r(\wt{v})=1$. Redefine $\wt{t_E}$ by
$\wt{v}\wt{t_E}$, then $\wt{t_E}$ is a lifting of $t_E$ and $\wt{s}|_{\wt{E}}=\wt{t_E}^{\otimes\mu}$ is $\mu$-divisible.
\end{proof}

\begin{cor}\label{4.5}
Let $X=\PP^n_k$ with $n\geq 3$, and $\LL$ an invertible sheaf on $X$.
Let $N$ be a positive integer prime to $p$, and $D$ an effective divisor
on $X$ with $\LL^N=\OO_X(D)$ and $\Sing(D_{\rm red})=\emptyset$.
Let $\pi:Y\ra X$ be the cyclic cover obtained by taking the $N$-th root out of $D$.
Then $Y$ is a smooth projective scheme which is strongly liftable over $W_2(k)$.
\end{cor}

\begin{proof}
Since the projective space $\PP^n_k\,(n\geq 3)$ satisfies the $H^i$-vanishing condition for $i=1,2$,
the conclusion follows from Corollary \ref{4.4}.
\end{proof}

By means of Corollary \ref{4.5}, we can construct many strongly liftable varieties of general type.

\begin{ex}\label{4.6}
Let $X=\PP^n_k$, $\LL=\OO_X(1)$ and $N$ a positive integer such that $n\geq 3$, $(N,p)=1$ and $N>n+2$.
Let $H$ be a general element in the linear system of $\OO_X(N)$. Then $H$ is a smooth irreducible
hypersurface of degree $N$ in $X$ with $\LL^N=\OO_X(H)$. Let $\pi:Y\ra X$ be the cyclic cover
obtained by taking the $N$-th root out of $H$. Then by Corollary \ref{4.5}, $Y$ is a strongly liftable
smooth projective variety. By Hurwitz's formula, we have $K_Y=\pi^*(K_X+\frac{N-1}{N}H)$.
Since the degree of $K_X+\frac{N-1}{N}H$ is $N-(n+2)>0$, $K_Y$ is an ample divisor on $Y$,
hence $Y$ is of general type.
\end{ex}

Obviously, the $H^i$-vanishing condition for $i=1,2$ is too strong to give more applications.
Although there are no further evidences besides Corollary \ref{4.5}, we would like to put forward
the following conjecture, i.e.\ cyclic covers over toric varieties should be strongly liftable
over $W_2(k)$, whereas the liftability has already been proved in \cite[Corollary 4.4]{xie11}.

\begin{conj}\label{4.7}
Let $X$ be a smooth projective toric variety, and $\LL$ an invertible sheaf
on $X$. Let $N$ be a positive integer prime to $p$, and $D$ an effective
divisor on $X$ with $\LL^N=\OO_X(D)$ and $\Sing(D_{\rm red})=\emptyset$.
Let $\pi:Y\ra X$ be the cyclic cover obtained by taking the $N$-th root
out of $D$. Then $Y$ is a smooth projective scheme which is strongly liftable
over $W_2(k)$.
\end{conj}

\textsc{School of Mathematical Sciences, Fudan University,
Shanghai 200433, China}

\textit{E-mail address}: \texttt{qhxie@fudan.edu.cn}


\small
\begin{thebibliography}{KMM87}

\bibitem[DI87]{di}
P. Deligne, L. Illusie,
Rel\`{e}vements modulo $p^2$ et d\'{e}composition du complexe de de Rham,
{\it Invent.\ Math.}, {\bf 89} (1987), 247--270.

\bibitem[EV92]{ev}
H. Esnault, E. Viehweg,
{\it Lectures on vanishing theorems},
DMV Seminar, vol.\ {\bf 20}, Birkh\"{a}user, 1992.

\bibitem[Ha77]{ha77}
R. Hartshorne, {\it Algebraic geometry}, Springer-Verlag, 1977.

\bibitem[Ma80]{ma80}
H. Matsumura,
{\it Commutative Algebra}, Benjamin, 1980.

\bibitem[Xie10]{xie10}
Q. Xie,
Strongly liftable schemes and the Kawamata-Viehweg vanishing
in positive characteristic,
{\it Math.\ Res.\ Lett.}, {\bf 17} (2010), 563--572.

\bibitem[Xie11]{xie11}
Q. Xie,
Strongly liftable schemes and the Kawamata-Viehweg vanishing
in positive characteristic II,
{\it Math.\ Res.\ Lett.}, {\bf 18} (2011), 315--328.

\bibitem[XW13]{xw13}
Q. Xie, J. Wu,
Strongly liftable schemes and the Kawamata-Viehweg vanishing
in positive characteristic III,
{\it J. Algebra}, {\bf 395} (2013), 12--23.

\end{thebibliography}
\end{document}